\documentclass{amsart}
\usepackage{comment,amsfonts,stmaryrd,tikz-cd,amssymb,amsmath,hyperref,appendix,enumerate,amsthm,bbm,mathrsfs,calligra}
\usepackage[style=alphabetic, maxnames=10]{biblatex}
\DeclareMathOperator{\M}{M}

\DeclareMathOperator{\GL}{GL}

\DeclareMathOperator{\Aut}{Aut}
\DeclareMathOperator{\op}{op}

\DeclareMathOperator{\Hom}{Hom}

\DeclareMathOperator{\Ext}{Ext}
\DeclareMathOperator{\Dir}{Dir}
\DeclareMathOperator{\length}{length}

\theoremstyle{plain}
\newtheorem{theorem}{Theorem}[section]
\newtheorem{theorem-definition}{Theorem-Definition}[section]

\newtheorem{corollary}[theorem]{Corollary}

\newtheorem{lemma}[theorem]{Lemma}
\newtheorem{proposition}[theorem]{Proposition}

\theoremstyle{definition}

\newtheorem{proposition-definition}[theorem]{Proposition-Definition}

\newtheorem{example}[theorem]{Example}

\newtheorem{scholium}[theorem]{Scholium}

\theoremstyle{remark}

\newtheorem{notation}[theorem]{Notation}

\newcommand\restr[2]{{
  \left.\kern-\nulldelimiterspace 
  #1 
  \right|_{#2} 
  }}

\title{Solomon zeta functions over arithmetic orders}
\author{Sean B. Lynch}

\newcommand{\Addresses}{{
  \bigskip
  \footnotesize
  \textsc{Sean B. Lynch, Institute for Theoretical Sciences, Westlake University, Hangzhou, Zhejiang, China}\par\nopagebreak
  \textit{E-mail address}: \texttt{sblynch@westlake.edu.cn}

}}

\begin{document}

\begin{abstract}
We prove an effective version of Solomon's first conjecture for lattices over orders in finite-dimensional semisimple algebras over nonarchimedean local fields. We express the quotient of a partial Solomon zeta function by the corresponding maximal-order zeta function as a finite sum whose terms are determined by finite module-theoretic data and weighted by polynomials defined using the M\"obius function of finite submodule posets. The resulting expression is independent of the chosen maximal overorder. Our proof is purely algebraic and is first formulated for the refined Bushnell--Reiner zeta functions. As an application, we obtain explicit formulas for the Solomon zeta functions of all lattices over $\mathbb{Z}_p[\mathbb{Z}/p\mathbb{Z}]$, including non-projective lattices.
\end{abstract}

\maketitle

\section{Introduction}

The Solomon zeta function of a lattice over an arithmetic order enumerates its finite-index submodules \cite{Solomon1977}. Hey gave an explicit formula for Solomon zeta functions over maximal orders \cite{Hey27}. The problem of computing Solomon zeta functions over non-maximal orders has attracted considerable attention \cite{Reiner1980,BushnellReiner1981,Hironaka1981,Hironaka1985,Denert1990,Hanaki2001,Wittmann2004,Akihide2016,Hofmann2016,Hirasaka2018,Babei&Herman2023,Espinosa2023}.

By an Euler product formula, the global problem reduces to the local problem. Moreover, a global order is non-maximal at only finitely many places. Thus, computing Solomon zeta functions in general amounts to determining, at each such place, the quotient of the local zeta function over the non-maximal order by the corresponding zeta function over a maximal order. Bushnell and Reiner proved Solomon's first conjecture, which states that each such quotient is a Dirichlet polynomial \cite{Solomon1979,BushnellReiner1980full}. However, their proof is not effective: although it establishes that the quotient is a polynomial, it does not give an independently computable expression for that polynomial and therefore does not provide an effective means of computing the Solomon zeta function.

The purpose of this paper is to prove an effective version of Solomon's first conjecture by expressing this quotient explicitly in terms of polynomials determined by finite module-theoretic data. In particular, our formula provides an effective method for computing the Solomon zeta function. Our proof of the effective version of Solomon's first conjecture is independent of Bushnell and Reiner's proof of the original conjecture.

Let $R$ be the ring of integers in a nonarchimedean local field $K$, with residue field $\mathbb{F}_q$. Let $\Lambda$ be an $R$-order in a finite-dimensional semisimple $K$-algebra $A$. Let $L$ and $M$ be full left $\Lambda$-lattices in a finitely generated left $A$-module $V$. Solomon introduced the following elements of the formal power series ring $\mathbb{Z}\llbracket q^{-s}\rrbracket$:
$$\zeta_\Lambda(L;s) = \sum_{X\le_\Lambda L} |L/X|^{-s} \qquad \text{and} \qquad \zeta_\Lambda(L, M;s) = \sum_{\substack{X\le_\Lambda L\\ X\simeq M}} |L/X|^{-s},$$
where the first sum runs over all finite-index submodules $X$ of $L$, while the second is restricted to those $X$ that are isomorphic to $M$. The first series is the \textit{Solomon zeta function} of $L$, and the second is the \textit{partial Solomon zeta function} of $L$ with respect to $M$. It follows from the Jordan--Zassenhaus theorem \cite[Thm.~24.7]{CurtisReiner1981} that $\zeta_\Lambda(L;s)=\sum_{[M]} \zeta_{\Lambda}(L,M;s)$, where $[M]$ ranges over the finitely many isomorphism classes of full left $\Lambda$-lattices in $V$.

Solomon used the M\"obius function of the locally finite partially ordered set of finite-index submodules of $L$ to prove that $\zeta_\Lambda(L,M;s)$ lies in $\mathbb{Q}(q^{-s})$, and hence $\zeta_\Lambda(L;s)$ does as well \cite{Solomon1977}. Solomon later posed two conjectures concerning these zeta functions \cite{Solomon1979}. Bushnell and Reiner proved the first using $p$-adic zeta integrals \cite{BushnellReiner1980full}; Iyama proved the second using rejective chains of lattice categories \cite{Iyama2003}. 

Let us now recall Solomon's first conjecture in the form proved by Bushnell and Reiner. Choose a maximal $R$-order $\Lambda'$ in $A$ containing $\Lambda$. Hey's formula expresses $\zeta_{\Lambda'}(\Lambda'L;s)$ as an explicit product of geometric series; see Corollary \ref{corollary_hey's_formula}. In particular, $\zeta_{\Lambda'}(\Lambda'L;s)$ is a unit in $\mathbb{Z}\llbracket q^{-s}\rrbracket$ that is independent of the choice of $\Lambda'$. Bushnell and Reiner then proved that the quotient series 
 $$\frac{\zeta_\Lambda(L,M;s)}{\zeta_{\Lambda'}(\Lambda'L;s)}$$
lies in the polynomial ring $\mathbb{Z}[q^{-s}]$. Their method begins by expressing $\zeta_\Lambda(L,M;s)$ as a $p$-adic zeta integral. They then decompose the resulting
Schwartz--Bruhat function and reduce the integral to a sum over Hermite
normal forms, which can be compared directly with Hey's formula. Unfortunately, the required Schwartz--Bruhat function decomposition is not effective.

The main result of this paper gives an effective formula for the polynomial $\zeta_\Lambda(L,M;s)/\zeta_{\Lambda'}(\Lambda'L;s)$. Our proof returns to the tool originally introduced by Solomon: the M\"obius function of a locally finite submodule poset.

\begin{theorem}\label{intro_effective_solomon_first_conjecture}
Let $R$ be the ring of integers in a nonarchimedean local field $K$, with residue field $\mathbb{F}_q$. Let $\Lambda$ be an $R$-order in a finite-dimensional semisimple $K$-algebra $A$. Let $L$ and $M$ be full left $\Lambda$-lattices in a finitely generated left $A$-module $V$. Choose a maximal $R$-order $\Lambda'$ in $A$ containing $\Lambda$. Set
$L'=\Lambda'L$ and $I=\{a\in A: \Lambda'a\Lambda'\subseteq \Lambda\}$. Let $\mu$ be the M\"obius function of the
poset of submodules of the finite left
$\Lambda'/I$-module $L'/IL'$. Then we have the formula
\begin{align*}
    \frac{\zeta_\Lambda(L,M;s)}{\zeta_{\Lambda'}(L';s)}
    =
    \sum_{\substack{X\le_\Lambda L\\ X\simeq M \\ IL'\subseteq \Lambda'X}}
    |L/X|^{-s}
    Q_{\Lambda'}(\Lambda' X/IL';s),
\end{align*}
where, for $W=\Lambda'X/IL'$, we define
$$Q_{\Lambda'}(W;s):=\sum_{U\leq_{\Lambda'} W}\mu(U,W)|W/U|^{-s}.$$
Since the outer sum is finite and each
$Q_{\Lambda'}(W;s)$ belongs to $\mathbb{Z}[q^{-s}]$, the quotient
$\zeta_\Lambda(L,M;s)/\zeta_{\Lambda'}(L';s)$ also belongs to
$\mathbb{Z}[q^{-s}]$.
\end{theorem}

We may interpret the above formula for $\zeta_\Lambda(L,M;s)/\zeta_{\Lambda'}(L';s)$ as follows. Multiplication by $\zeta_{\Lambda'}(L';s)^{-1}$ restricts the defining sum of $\zeta_\Lambda(L,M;s)$ to those $X$ satisfying $IL'\subseteq\Lambda' X$ and weights each remaining term by $Q_{\Lambda'}(\Lambda'X/IL';s)$. Although the restriction and the individual weights depend on the chosen maximal overorder, the resulting weighted sum is independent of that choice. Scholium \ref{scholium_effective_formula} and Corollary \ref{corollary_last_effective_formula} give alternative interpretations better suited to explicit computation.

As an application of Corollary \ref{corollary_last_effective_formula}, we derive an explicit formula for the Solomon zeta function of every lattice over the group algebra
$\Lambda=\mathbb{Z}_p[\mathbb{Z}/p\mathbb{Z}]\simeq\mathbb{Z}_p\times_{\mathbb{F}_p}\mathbb{Z}_p[\omega]$,
where $p$ is a prime and $\omega$ is a primitive $p$-th root of unity. Our formula recovers the previously known cases and yields new formulae for the remaining lattices. The computation of Solomon zeta functions in this setting has a rich history. In his seminal paper, Solomon proved \cite{Solomon1977} that
$$
\zeta_\Lambda(\Lambda;s)=\frac{1-p^{-s}+p^{1-2s}}{(1-p^{-s})^2}.
$$
Reiner later recovered this formula using a method for parameterising left ideals in certain fibre-product rings \cite{Reiner1980}. Bushnell and Reiner subsequently obtained it using a $p$-adic integral method \cite{BushnellReiner1980full}. Using Morita equivalence, Wittmann reduced the computation of the Solomon zeta function of $\Lambda^{\oplus c}$ to the enumeration of left ideals in $M_c(\Lambda)$ and then applied Reiner's fibre-product method \cite{Wittmann2004}.

These methods do not appear to extend directly to non-projective
$\Lambda$-lattices. There are precisely three indecomposable
$\Lambda$-lattices up to isomorphism: $\mathbb{Z}_p$,
$\mathbb{Z}_p[\omega]$, and $\Lambda$
\cite[\S.~2]{HellerReiner1962}. Neither $\mathbb{Z}_p$ nor
$\mathbb{Z}_p[\omega]$ is projective as a $\Lambda$-module. Together
with the Krull--Schmidt--Azumaya theorem
\cite[Thm.~30.6]{CurtisReiner1981}, this classifies all
$\Lambda$-lattices up to isomorphism.

\begin{theorem}\label{intro_group_algebra}
Let $\Lambda=\mathbb{Z}_p[\mathbb{Z}/p\mathbb{Z}]\simeq\mathbb{Z}_p\times_{\mathbb{F}_p}\mathbb{Z}_p[\omega]$, where $p$ is a prime and $\omega$ is a primitive $p$-th root of unity. Let $L$ be a $\Lambda$-lattice, and let $a,b,c$ be the unique nonnegative integers such that $$L\simeq \mathbb{Z}_p^{\oplus a}\oplus\mathbb{Z}_p[\omega]^{\oplus b}\oplus\Lambda^{\oplus c}.$$ Let $\Lambda'\simeq\mathbb{Z}_p\times\mathbb{Z}_p[\omega]$ be the maximal $\mathbb{Z}_p$-order in $\mathbb{Q}_p[\mathbb{Z}/p\mathbb{Z}]$, and write $\binom{\alpha}{\beta}_p$ for the Gaussian binomial coefficient enumerating the $\beta$-dimensional subspaces of $\mathbb{F}_p^{\oplus\alpha}$. Then we have
$$\zeta_{\Lambda'}(\Lambda'L;s)=\prod_{h=0}^{a+c-1}(1-p^{h-s})^{-1}\prod_{i=0}^{b+c-1}(1-p^{i-s})^{-1},$$
and
    $$\frac{\zeta_\Lambda(L;s)}{\zeta_{\Lambda'}(\Lambda'L;s)}=\sum_{j=0}^c \binom{c}{j}_p p^{-(c-j)s}P_j^{a,b,c}(p,s)\prod_{k=0}^{j-1}(1-p^{k-s})^2,$$
    where
    $$P_j^{a,b,c}(p,s)=\sum_{\ell=0}^{\min(a+c-j,b+c-j)}\binom{a+c-j}{\ell}_p\binom{b+c-j}{\ell}_p |\GL_\ell(\mathbb{F}_p)|p^{\ell(j-s)}.$$
\end{theorem}

We derive this theorem as a corollary of a more general result for lattices over orders satisfying certain Artin--Wedderburn conditions (Theorem \ref{theorem_application}).

More generally, we work throughout the paper with the refined Bushnell--Reiner zeta functions \cite{BushnellReiner1987} and deduce the corresponding results for Solomon zeta functions as corollaries. We introduce these zeta functions in Section \ref{section_br_zeta}.

In Section \ref{section_hey}, we prove that the $\Lambda$-composition factors of simple left $\Lambda'$-modules are independent of the choice of maximal overorder $\Lambda'\supseteq\Lambda$ (Proposition \ref{prop_independence_overorder}). The result and its proof are inspired by Bushnell and Reiner's argument \cite[Cor.~3.16]{BushnellReiner1987}. We streamline their method so that it is entirely algebraic and applies in a more general setting. The resulting proof relies on the structure theory of maximal orders \cite[Ch.~3 \& 5]{Reiner2003}. At the end of the section, we use this independence result to establish the version of Hey's formula required for our treatment of Solomon's first conjecture for Bushnell--Reiner zeta functions (Theorem \ref{bushnell-reiner-hey_theorem}). In particular, we invoke our abstract version of Hey's formula \cite[Prop.~A.1]{Lynch2026}, whose proof uses Solomon's M\"obius function method. 

We prove our main theorem (Theorem \ref{effective_solomons_first_conjecture}) in Section \ref{section_solomons_conjecture}. We begin with an elementary identity for composition factors, prove an invariance property of an auxiliary polynomial, and then use Möbius inversion to evaluate a second auxiliary sum. Both propositions rely on the structure theory of maximal orders, and the latter also uses M\"obius inversion. After proving the main theorem, we give a scholium and a corollary that are better suited to explicit computation.

Finally, in Section \ref{section_application}, we apply the effective formula to orders satisfying the Artin--Wedderburn conditions described above and derive the explicit formula for lattices over $\mathbb{Z}_p[\mathbb{Z}/p\mathbb{Z}]$.

\section{Dirichlet rings and Bushnell--Reiner zeta functions}\label{section_br_zeta}

\subsection{Dirichlet rings}

Let $R$ be a commutative semilocal ring.

Let $\Lambda$ be a module-finite $R$-algebra. Then $\Lambda$ is semilocal, and thus the set $\mathcal{S}(\Lambda)$ of isomorphism classes of simple left $\Lambda$-modules is finite. By the Jordan--H\"older theorem, the Grothendieck group $\mathcal{G}(\Lambda)$ of the category of finite-length left $\Lambda$-modules is the free abelian group on $\mathcal{S}(\Lambda)$. We write the group $\mathcal{G}(\Lambda)$ multiplicatively, so that $[M]=[M'][M'']$ for every short exact sequence $0\to M'\to M\to M''\to 0$ of finite length left $\Lambda$-modules. Let $\mathcal{M}(\Lambda)$ be the submonoid of $\mathcal{G}(\Lambda)$ generated by $\mathcal{S}(\Lambda)$. Thus $\mathbb{Z}[\mathcal{M}(\Lambda)]$ is naturally identified with the polynomial ring over $\mathbb{Z}$ in the variables indexed by $\mathcal{S}(\Lambda)$. We define the \textit{Dirichlet ring} $\Dir(\Lambda)$ as the completion of $\mathbb{Z}[\mathcal{M}(\Lambda)]$ with respect to the ideal generated by $\mathcal{S}(\Lambda)$. Then $\Dir(\Lambda)$ is naturally identified with the formal power series ring over $\mathbb{Z}$ in the variables indexed by $\mathcal{S}(\Lambda)$.

Both $\mathcal{M}$ and $\Dir$ are contravariant in $\Lambda$. Indeed, if $\varphi:\Lambda\to\Gamma$ is a homomorphism of module-finite
$R$-algebras, then restriction of scalars determines a monoid homomorphism
$$
\varphi^*:\mathcal{M}(\Gamma)\longrightarrow\mathcal{M}(\Lambda),
\qquad
[{}_\Gamma S]\longmapsto[{}_\Lambda S],
$$
on the simple generators $S$ of $\mathcal{M}(\Gamma)$. This is well-defined because every simple left $\Gamma$-module has finite length as an $R$-module and therefore, since every $\Lambda$-submodule is an $R$-submodule, also has finite length as a left $\Lambda$-module. Passing to completed monoid rings gives a continuous ring homomorphism $\varphi^*:\Dir(\Gamma)\longrightarrow \Dir(\Lambda)$. Thus $\Dir$ is a contravariant functor from the category of module-finite $R$-algebras to the category of topological rings.

\subsection{Bushnell--Reiner zeta functions}

Let $R$ be a commutative local ring whose maximal ideal is finitely generated and whose residue field is $\mathbb{F}_q$. 

Let $\Lambda$ be a module-finite $R$-algebra, and let $L$ and $M$ be finitely generated left $\Lambda$-modules. For each $g\in\mathcal{M}(\Lambda)$, there are only finitely many finite-index left $\Lambda$-submodules $X$ of $L$ satisfying $[L/X]=g$ \cite[p.~5]{Lynch2026}. Thus we may define the following elements of $\Dir(\Lambda)$:
$$Z_\Lambda(L)=\sum_{X\leq_\Lambda L} [L/X]
    \qquad \text{and} \qquad
    Z_\Lambda(L,M)=
    \sum_{\substack{X\leq_\Lambda L\\ X\simeq M}}[L/X],$$
where the first sum runs over all finite-index $\Lambda$-submodules $X$ of $L$, while the second is restricted to those $X$ that are isomorphic to $M$. The first series is the \textit{Bushnell--Reiner zeta function} of $L$, and the second is the \textit{partial Bushnell--Reiner zeta function} of $L$ with respect to $M$. 

Let $\varphi:R\to\Lambda$ be the structure homomorphism. Since $\mathcal{S}(R)=\{[\mathbb{F}_q]\}$, we have $\Dir(R)=\mathbb{Z}\llbracket [\mathbb{F}_q]\rrbracket$. There is therefore a topological ring isomorphism 
$$\psi:\Dir(R)\xrightarrow{\sim} \mathbb{Z}\llbracket q^{-s}\rrbracket, \qquad [\mathbb{F}_q]\mapsto q^{-s}.$$ 
Applying the continuous ring homomorphism $\psi\circ\varphi^*:\Dir(\Lambda)\to\mathbb{Z}\llbracket q^{-s}\rrbracket$ to the Bushnell--Reiner zeta functions defined above recovers the Solomon zeta functions $\zeta_\Lambda(L;s)$ and $\zeta_\Lambda(L,M;s)$ defined in the introduction, although here we are working in a slightly more general setting.

\section{Hey's formula and independence of the maximal overorder}\label{section_hey}

\begin{notation}
Let $R$ be a complete discrete valuation ring with field of fractions $K$. Let $\Lambda$ be an $R$-order in a finite-dimensional semisimple $K$-algebra $A$, and let $V$ be a finitely generated left $A$-module. For any two full left $\Lambda$-lattices $L$ and $M$ in $V$, we may define 
$$[L:M]:=\left[\frac{L}{L\cap M}\right]\left[\frac{M}{L\cap M}\right]^{-1}\in \mathcal{G}(\Lambda).$$
Then, for three full left $\Lambda$-lattices $L$, $M$, and $N$ in $V$, we have 
$$[L:M]=[M:L]^{-1} \qquad \text{and} \qquad [L:N]=[L:M][M:N].$$
Similarly, for any two full left $\Lambda$-lattices $L$ and $M$ in $V$, we may define 
$$\length_\Lambda(L:M):=\length_\Lambda\left(\frac{L}{L\cap M}\right)-\length_\Lambda\left(\frac{M}{L\cap M}\right)\in \mathbb{Z}.$$
\end{notation}

\begin{lemma}
      Let $R$ be a complete discrete valuation ring with field of fractions $K$. Let $\Lambda$ be an $R$-order in a finite-dimensional semisimple $K$-algebra $A$. Let $L$ and $M$ be full left $\Lambda$-lattices in a finitely generated left $A$-module $V$. Let $x,y\in \Aut_A(V)^{\op}$. Then we have the following identities in the multiplicative group $\mathcal{G}(\Lambda)$:
      $$[L:Ly]=[M:My] \qquad \text{and} \qquad [L:Lxy]=[L:Lx][L:Ly].$$
\end{lemma}
\begin{proof}
    For the first identity, we compute that 
    $$[L:Ly][Ly:My]=[L:My]=[L:M][M:My].$$
    Moreover, since right multiplication by $y$ on $V$ induces left $\Lambda$-module isomorphisms 
    $$\frac{L}{L\cap M}\simeq \frac{Ly}{Ly\cap My} \qquad \text{and} \qquad \frac{M}{L\cap M}\simeq \frac{My}{Ly\cap My},$$
    we have 
    $$[L:M]=[Ly:My].$$
    Hence,
    $$[L:Ly]=[L:M][M:My][Ly:My]^{-1}=[M:My].$$

    For the second identity, we start with
    $$[L:Lxy]=[L:Lx][Lx:Lxy].$$
    Taking $M=Lx$ in the previous identity, we have $[Lx:Lxy]=[L:Ly]$. Substituting this back yields $[L:Lxy]=[L:Lx][L:Ly]$, completing the proof.
\end{proof}

\begin{proposition}\label{prop_independence_overorder}
Let $R$ be a complete discrete valuation ring with field of fractions $K$. Let $\Lambda$ be an $R$-order in a finite-dimensional semisimple $K$-algebra $A$. Let $\Lambda'$ and $\Lambda''$ be maximal $R$-orders in $A$ such that $\Lambda'\supseteq\Lambda$ and $\Lambda''\supseteq\Lambda$. Let $(e_i)$ be a complete set of primitive central idempotents of $A$ (so that $1 = \sum e_i$ is the decomposition of $1$ into primitive orthogonal central idempotents). Put $\Lambda'_i=\Lambda'e_i$ and $\Lambda''_i=\Lambda''e_i$. Let $S'_i$ be a simple left $\Lambda'_i$-module and $S''_i$ be a simple left $\Lambda''_i$-module. Then $[S'_i]=[S''_i]$ in $\mathcal{G}(\Lambda)$.
\end{proposition}
\begin{proof}
Let $A_i=Ae_i$. Choose $x_i\in A_i^\times\cap\Lambda_i'$ such that $\Lambda_i'/\Lambda_i'x_i\simeq S'_i$; such an element exists by the structure theory of maximal orders. By the previous lemma, applied to the two full left $\Lambda$-lattices $\Lambda'_i$ and $\Lambda_i''$ in $A_i$, we have
$$[\Lambda_i':\Lambda_i' x_i]=[\Lambda_i'':\Lambda_i'' x_i].$$
Because $S''_i$ is, up to isomorphism, the unique simple left $\Lambda_i''$-module, we have 
$$[\Lambda_i'':\Lambda_i''x_i]=[S_i'']^{m_i}, \qquad m_i=\length_{\Lambda_i''}(\Lambda_i'':\Lambda_i'' x_i).$$
Therefore, 
$$[S_i']=[S_i'']^{m_i},$$
and it remains to prove that $m_i=1$.

Applying the previous lemma again, this time regarding $\Lambda_i'$ and $\Lambda_i''$ as full $R$-lattices in the $K$-vector space $A_i$, we obtain
$$\length_R(\Lambda_i':\Lambda_i'x_i)=\length_R(\Lambda''_i:\Lambda''_i x_i).$$
The left-hand side is $\length_R(S_i')$. Since $S_i''$ is the unique simple left $\Lambda_i''$-module up to isomorphism, the right-hand side is $m_i\length_R(S_i'')$. Thus 
$$\length_R(S_i')=m_i\length_R(S_i'').$$
Finally, the maximal orders $\Lambda'_i$ and $\Lambda''_i$ are conjugate in $A_i$. Conjugation gives an $R$-algebra isomorphism between them and hence preserves the $R$-lengths of their simple modules. Therefore, $\length_R(S'_i)=\length_R(S_i'')$, so $m_i=1$, as desired.
\end{proof}

\begin{theorem}\label{bushnell-reiner-hey_theorem}
    Let $R$ be the ring of integers in a nonarchimedean local field $K$, with residue field $\mathbb{F}_q$. Let $\Lambda$ be an $R$-order in a finite-dimensional semisimple $K$-algebra $A$. Let $L$ be a full left $\Lambda$-lattice in a finitely generated left $A$-module $V$. 
    
    Let $(e_i)_{i\in I}$ be a complete set of primitive central idempotents of $A$. Put $A_i=Ae_i$, $V_i=A_iV$ and $\ell_i=\length_{A_i}(V_i)$. Write the Artin--Wedderburn decomposition $A_i\simeq\M_{n_i}(D_i)$ where $n_i$ is a positive integer and $D_i$ is a finite-dimensional division $K$-algebra. Let $\Delta_i$ be the unique maximal $R$-order in $D_i$, and put $q_i=|\Delta_i/J(\Delta_i)|$.
    
    Choose a maximal $R$-order $\Lambda'$ in $A$ containing $\Lambda$, and let $\iota:\Lambda\hookrightarrow\Lambda'$ be the inclusion. Put $L'=\Lambda'L$ and $\Lambda'_i=\Lambda'e_i$. Let $\lambda'_i\in \mathcal{S}(\Lambda')$ denote the unique isomorphism class of simple left $\Lambda'$-module supported on the $i$-th component. Then we have 
    $$\iota^*Z_{\Lambda'}(L')=\prod_{i\in I}\prod_{j=0}^{\ell_i-1}(1-q^j_i \iota^*\lambda'_i)^{-1}.$$ 
    In particular, $\iota^*Z_{\Lambda'}(L')$ is a unit in $\Dir(\Lambda)$ that is independent of the choice of $\Lambda'$. 
\end{theorem}
\begin{proof}
By the structure theory of maximal $R$-orders, every finite-index $\Lambda'$-submodule of $L'$ is isomorphic to $L'$, and 
    $$\Lambda'/J(\Lambda')\simeq \prod_{i\in I}\M_{n_i}(\mathbb{F}_{q_i}), \qquad L'/J(\Lambda')L'\simeq \bigoplus_{i\in I}(\mathbb{F}_{q_i}^{\oplus n_i})^{\oplus\ell_i}.$$
    The abstract form of Hey's formula \cite[Prop.~A.1]{Lynch2026} therefore gives
    $$Z_{\Lambda'}(L')=\prod_{i\in I}\prod_{j=0}^{\ell_i-1}(1-q^j_i \lambda'_i)^{-1}.$$
    Applying $\iota^*$ yields 
    $$\iota^*Z_{\Lambda'}(L')=\prod_{i\in I}\prod_{j=0}^{\ell_i-1}(1-q^j_i \iota^*\lambda'_i)^{-1}.$$ 
Proposition \ref{prop_independence_overorder} shows that $\iota^*\lambda'_i$ is independent of the chosen maximal overorder $\Lambda'$. Hence the entire product is independent of $\Lambda'$. 
\end{proof}

By applying the topological ring homomorphism $\Dir(\Lambda)\to\mathbb{Z}\llbracket q^{-s}\rrbracket: [S]\mapsto|S|^{-s}$, we recover the following statement of Hey's formula as a corollary of Theorem \ref{bushnell-reiner-hey_theorem}. 

\begin{corollary}\label{corollary_hey's_formula}
    Keep the assumptions and notation of Theorem \ref{bushnell-reiner-hey_theorem}. Then 
    $$\zeta_{\Lambda'}(L';s)=\prod_{i\in I}\prod_{j=0}^{\ell_i-1}(1-q^{j-n_is}_i)^{-1}$$
    is a unit in $\mathbb{Z}\llbracket q^{-s}\rrbracket$, and it is independent of the choice of $\Lambda'$.
\end{corollary}

\section{An effective version of Solomon's first conjecture}\label{section_solomons_conjecture}

Let $R$ be the ring of integers in a nonarchimedean local field $K$, with residue field $\mathbb{F}_q$. Let $\Lambda$ be an $R$-order in a finite-dimensional semisimple $K$-algebra $A$. Let $L$ and $M$ be full left $\Lambda$-lattices in a finitely generated left $A$-module $V$. Choose a maximal $R$-order $\Lambda'$ in $A$ such that $\Lambda'\supseteq\Lambda$, and let $\iota:\Lambda\hookrightarrow\Lambda'$ be the inclusion. Write $L'=\Lambda'L$ and $I=\{a\in A: \Lambda'a\Lambda'\subseteq \Lambda\}$. Let $\mu$ be the M\"obius function of the partially ordered set of left $\Lambda'/I$-submodules of $L'/IL'$ \cite[\S.~8.6]{JacobsonBasicAlgebraI}. 

\begin{lemma}\label{only_lemma}
    Let $X$ be a finite-index $\Lambda$-submodule of $L$, let $Y=\Lambda'X$, and let $Z=Y+IL'$. Then we have 
    $$[L/X]=[L/(L\cap Z)][Z/Y][(L\cap Y)/X]$$
    in $\mathcal{M}(\Lambda)$.
\end{lemma}
\begin{proof}
    Since $L\supseteq L\cap Z\supseteq L\cap Y\supseteq X$, we have 
    $$[L/X]=[L/(L\cap Z)][(L\cap Z)/(L\cap Y)][(L\cap Y)/X].$$
    So, it suffices to show that $(L\cap Z)/(L\cap Y)\simeq Z/Y$ as left $\Lambda$-modules. To this end, we observe that $(L\cap Z)+Y\subseteq Z$ and $(L\cap Z)+Y\supseteq IL'+Y=Z$, so that $(L\cap Z)+Y=Z$. Therefore, we have
    $$\frac{L\cap Z}{L\cap Y}=\frac{L\cap Z}{(L\cap Z)\cap Y}\cong \frac{(L\cap Z)+Y}{Y}=Z/Y,$$
    as desired.
\end{proof}

\begin{proposition}\label{first_prop}
    Let $Y$ be a finite-index $\Lambda'$-submodule of $L'$. Define
        $$P_{\iota}(L,M;Y):=\sum_{\substack{X\leq_\Lambda L\\ X\simeq M\\ \Lambda'X=Y}}[(L\cap Y)/X].$$
    Then the defining sum is finite, so $P_\iota(L,M;Y)\in\mathbb{Z}[\mathcal{M}(\Lambda)]$. Moreover,
        $$P_{\iota}(L,M;Y)=P_{\iota}(L,M;Y+IL').$$
\end{proposition}
\begin{proof}
If $X$ is a $\Lambda$-submodule of $L$ satisfying $\Lambda'X=Y$, then we have $IY= I\Lambda'X =IX\subseteq X\subseteq L\cap Y$. Moreover, the quotient module $Y/IY$ is finite. Hence, there are only finitely many terms in the sum defining $P_{\iota}(L,M;Y)$, and so it lies in $\mathbb{Z}[\mathcal{M}(\Lambda)]$.

If $Z$ is another finite-index $\Lambda'$-submodule of $L'$ for which there exists a left $\Lambda'$-module isomorphism $\phi: Y\xrightarrow{\sim} Z$ satisfying $\phi(L\cap Y)=L\cap Z$, then we have 
$$P_{\iota}(L,M;Y)=P_{\iota}(L,M;Z).$$
So, it suffices to show that we have such an isomorphism for $Z=Y+IL'$. In fact, we claim that there exists a left $\Lambda'$-module isomorphism $\phi:Y\xrightarrow{\sim} Y+IL'$ such that $\phi(y)\equiv y \bmod IL'$ for all $y\in Y$. If this were true, then we would have $\phi(L\cap Y)=L\cap (Y+IL')$ because $IL'\subseteq L$.

Our construction is componentwise with respect to the simple factors of the semisimple $K$-algebra $A$, and after applying the standard Morita equivalence for a maximal order in a simple algebra, it is enough to treat the case where $\Lambda'=\Delta$ is the maximal $R$-order in a finite-dimensional division $K$-algebra $A=D$. Choose a uniformiser $\pi$ for $\Delta$. Then there exists a nonnegative integer $c$ such that $I=\Delta\pi^c=\pi^c\Delta$. Moreover, by Smith normal form over $\Delta$ \cite[Thm.~17.7]{Reiner2003}, there exist a left $\Delta$-basis $v_1,\ldots,v_n$ of $L'$ and nonnegative integers $a_1,\ldots,a_n$ such that $Y=\bigoplus_j \Delta \pi^{a_j} v_j$. Therefore, $Y+IL'=\bigoplus_j \Delta \pi^{\min(a_j,c)} v_j$. Hence, we may define our desired left $\Delta$-module isomorphism $\phi:Y\xrightarrow{\sim} Y+IL'$ by taking $\pi^{a_j} v_j\mapsto \pi^{\min(a_j,c)} v_j$. If $a_j\le c$, then $\phi(\pi^{a_j}v_j)=\pi^{a_j}v_j$. If $a_j>c$, then $\phi(\pi^{a_j}v_j)-\pi^{a_j}v_j=(\pi^c-\pi^{a_j})v_j\in\pi^cL'=IL'$. Therefore, $\phi(y)\equiv y\bmod IL'$ for every $y\in Y$, as required.
\end{proof}

\begin{proposition}\label{second_prop}
Let $Z$ be a $\Lambda'$-submodule of $L'$ containing $IL'$ and let $W=Z/IL'$. Define 
    $$Q_{\Lambda'}(W):=\sum_{U\le_{\Lambda'} W}\mu(U,W)[{}_{\Lambda'}W/U]\in\mathbb{Z}[\mathcal{M}(\Lambda')].$$
Then we have
    $$\sum_{\substack{Y\le {}_{\Lambda'}L'\\ Y+IL'=Z}}[Z/Y]=\iota^*\left(Q_{\Lambda'}(W) Z_{\Lambda'}(L')\right).$$
\end{proposition}
\begin{proof}
For each $\Lambda'$-submodule $U$ of $W$, we define $$F(U):=\sum_{\substack{Y\le {}_{\Lambda'}L'\\ \frac{Y+IL'}{IL'}=U}}[Z/Y], \qquad G(U):=\sum_{\substack{Y\le {}_{\Lambda'}L'\\ \frac{Y+IL'}{IL'}\subseteq U}}[Z/Y].$$
Note that we have
$$G(U)=\sum_{T\le_{\Lambda'} U} F(T)$$
for each $U$. Therefore, by M\"obius inversion \cite[Cor.~1, p.~483]{JacobsonBasicAlgebraI}, we also have
$$F(U)=\sum_{T\le_{\Lambda'} U} \mu(T,U) G(T)$$
for all $U$. In particular,
$$\sum_{\substack{Y\le {}_{\Lambda'}L'\\ Y+IL'=Z}}[Z/Y]=F(W)=\sum_{U\le_{\Lambda'} W} \mu(U,W) G(U).$$
Hence, it now suffices to show that 
$$G(U)=[W/U]\iota^* Z_{\Lambda'}(L')$$
for all $U$.

To this end, fix $U$ and let $Z'$ be the left $\Lambda'$-submodule of $L'$ containing $IL'$ such that $Z'/IL'=U$. Then we have 
$$G(U)=\sum_{\substack{Y\le {}_{\Lambda'}L'\\ \frac{Y+IL'}{IL'}\subseteq U}}[Z/Y]=\sum_{Y\le {}_{\Lambda'}Z'}[Z/Y].$$
If $Y$ is a finite-index $\Lambda'$-submodule of $Z'$, then $Z\supseteq Z'\supseteq Y$ and $W/U\cong Z/Z'$, so that
$$[Z/Y]=[Z/Z'][Z'/Y]=[W/U][Z'/Y].$$
Substituting this into our last formula for $G(U)$, we find that 
$$G(U)=[W/U]\sum_{Y\le {}_{\Lambda'}Z'}[Z'/Y]=[W/U]\iota^* Z_{\Lambda'}(L').$$
Indeed, the last equality follows from the fact that $Z'\simeq L'$ as left $\Lambda'$-modules.
\end{proof}

We now come to our main theorem.

\begin{theorem}\label{effective_solomons_first_conjecture}
Let $R$ be the ring of integers in a nonarchimedean local field $K$, with residue field $\mathbb{F}_q$. Let $\Lambda$ be an $R$-order in a finite-dimensional semisimple $K$-algebra $A$. Let $L$ and $M$ be full left $\Lambda$-lattices in a finitely generated left $A$-module $V$. Choose a maximal $R$-order $\Lambda'$ in $A$ containing $\Lambda$, and let $\iota:\Lambda\hookrightarrow\Lambda'$ be the inclusion. Write $L'=\Lambda'L$ and $I=\{a\in A: \Lambda'a\Lambda'\subseteq \Lambda\}$. Let $\mu$ be the M\"obius function of the partially ordered set of left $\Lambda'/I$-submodules of $L'/IL'$. 

Then $\iota^*Z_{\Lambda'}(L')$ is a unit in $\Dir(\Lambda)$ that is independent of our choice of $\Lambda'$. Moreover, we have the formula
    $$\frac{Z_\Lambda(L,M)}{\iota^* Z_{\Lambda'}(L')}=\sum_{\substack{X\leq_{\Lambda} L\\ X\simeq M\\ IL'\subseteq\Lambda'X}}[L/X]\iota^*Q_{\Lambda'}(\Lambda'X/IL'),$$
where, with $W=\Lambda'X/IL'$,
$$Q_{\Lambda'}(W):=\sum_{U\leq_{\Lambda'} W}\mu(U,W)[{}_{\Lambda'}W/U]\in\mathbb{Z}[\mathcal{M}(\Lambda')].$$
Since the sum over $X$ is finite and
$\iota^*Q_{\Lambda'}(W)\in\mathbb{Z}[\mathcal{M}(\Lambda)]$, it follows that the quotient series $Z_\Lambda(L,M)/\iota^* Z_{\Lambda'}(L')$ lies in $\mathbb{Z}[\mathcal{M}(\Lambda)]$.
\end{theorem}
\begin{proof}
    The first assertion follows from Theorem \ref{bushnell-reiner-hey_theorem}: $\iota^*Z_{\Lambda'}(L')$ is a unit in $\Dir(\Lambda)$ that is independent of our choice of $\Lambda'$.

Now, we write
\begin{align*}
    Z_\Lambda(L,M)=\sum_{\substack{X\le_\Lambda L\\ X\simeq M}}[L/X]=\sum_{\substack{Y\le_{\Lambda'}L'}}\sum_{\substack{X\le_\Lambda L\\ X\simeq M\\ \Lambda'X=Y}}[L/X]=\sum_{\substack{Z\le_{\Lambda'} L'\\ IL'\subseteq Z}}\sum_{\substack{Y\le_{\Lambda'}L'\\ Y+IL'=Z}}\sum_{\substack{X\le_\Lambda L\\ X\simeq M\\ \Lambda'X=Y}}[L/X].
\end{align*}
So, by Lemma \ref{only_lemma}, we have
\begin{align*}
    Z_\Lambda(L,M)&=\sum_{\substack{Z\le_{\Lambda'} L'\\ IL'\subseteq Z}}[L/(L\cap Z)]\sum_{\substack{Y\le_{\Lambda'}L'\\ Y+IL'=Z}}[Z/Y]\sum_{\substack{X\le_\Lambda L\\ X\simeq M\\ \Lambda'X=Y}}[(L\cap Y)/X]\\
    &=\sum_{\substack{Z\le_{\Lambda'} L'\\ IL'\subseteq Z}}[L/(L\cap Z)]\sum_{\substack{Y\le_{\Lambda'}L'\\ Y+IL'=Z}}[Z/Y]P_\iota(L,M;Y).
\end{align*}
Thus, by Propositions \ref{first_prop} and \ref{second_prop}, we have
\begin{align*}
    Z_\Lambda(L,M)&=\sum_{\substack{Z\le_{\Lambda'} L'\\ IL'\subseteq Z}}[L/(L\cap Z)]\sum_{\substack{Y\le_{\Lambda'}L'\\ Y+IL'=Z}}[Z/Y]P_\iota(L,M;Z)\\
    &=\sum_{\substack{Z\le_{\Lambda'} L'\\ IL'\subseteq Z}}[L/(L\cap Z)]P_\iota(L,M;Z)\sum_{\substack{Y\le_{\Lambda'}L'\\ Y+IL'=Z}}[Z/Y]\\
    &=\sum_{\substack{Z\le_{\Lambda'} L'\\ IL'\subseteq Z}}[L/(L\cap Z)]P_\iota(L,M;Z)\iota^*\left(Q_{\Lambda'}(Z/IL') Z_{\Lambda'}(L')\right)\\
    &=\left(\iota^* Z_{\Lambda'}(L')\right)\sum_{\substack{Z\le_{\Lambda'} L'\\ IL'\subseteq Z}}[L/(L\cap Z)]P_\iota(L,M;Z)\iota^*Q_{\Lambda'}(Z/IL').
\end{align*}
Hence, dividing by $\iota^* Z_{\Lambda'}(L')$ and expanding $P_\iota(L,M;Z)$, we compute
\begin{align*}
    \frac{Z_\Lambda(L,M)}{\iota^* Z_{\Lambda'}(L')}&=\sum_{\substack{Z\le_{\Lambda'} L'\\ IL'\subseteq Z}}[L/(L\cap Z)]P_\iota(L,M;Z)\iota^*Q_{\Lambda'}(Z/IL')\\
    &=\sum_{\substack{Z\le_{\Lambda'} L'\\ IL'\subseteq Z}}[L/(L\cap Z)]\sum_{\substack{X\leq_\Lambda L\\ X\simeq M\\ \Lambda'X=Z}}[(L\cap Z)/X]\iota^*Q_{\Lambda'}(Z/IL')\\
    &=\sum_{\substack{Z\le_{\Lambda'} L'\\ IL'\subseteq Z}}\sum_{\substack{X\leq_\Lambda L\\ X\simeq M\\ \Lambda'X=Z}}[L/X]\iota^*Q_{\Lambda'}(Z/IL')\\
    &=\sum_{\substack{X\leq_{\Lambda} L\\ X\simeq M\\ IL'\subseteq\Lambda'X}}[L/X]\iota^*Q_{\Lambda'}(\Lambda'X/IL').
\end{align*}
Finally, if $IL'\subseteq\Lambda' X$, then $I^2L=I^2L'\subseteq I\Lambda'X=IX\subseteq X$. Since $L/I^2L$ is finite, only finitely many such submodules $X$ exist.
\end{proof}

Now, we see that we may obtain Theorem \ref{intro_effective_solomon_first_conjecture} as a corollary of Theorem \ref{effective_solomons_first_conjecture} by applying the topological ring homomorphism $\Dir(\Lambda)\to\mathbb{Z}\llbracket q^{-s}\rrbracket: [S]\mapsto |S|^{-s}$.

For explicit computation, it is generally preferable to retain the expression involving $P_\iota(L,M;Z)$, rather than expand it into a single sum over $X$. Indeed, multiplying $IL'\subseteq \Lambda'X$ on the left by $I$, we obtain $I^2L=I^2L'\subseteq I\Lambda'X=IX\subseteq X$. Thus, the sum over $X$ essentially requires us to work modulo $I^2$. However, we can break this sum into two steps, in each of which we work modulo $I$.

\begin{scholium}\label{scholium_effective_formula}
Keep the assumptions and notation of Theorem \ref{effective_solomons_first_conjecture}. Then we have the formulae
\begin{align*}
    \frac{Z_\Lambda(L,M)}{\iota^*Z_{\Lambda'}(L')}&=\sum_{\substack{Z\le_{\Lambda'} L'\\ IL'\subseteq Z}}[L/(L\cap Z)]P_\iota(L,M;Z)\iota^*Q_{\Lambda'}(Z/IL'),\\
    \frac{\zeta_\Lambda(L,M;s)}{\zeta_{\Lambda'}(L';s)}&=\sum_{\substack{Z\le_{\Lambda'} L'\\ IL'\subseteq Z}}|L/(L\cap Z)|^{-s}P_\iota(L,M;Z;s)Q_{\Lambda'}(Z/IL';s)
\end{align*}
where, for $W=Z/IL'$, we define
\begin{gather*}
    P_{\iota}(L,M;Z):=\sum_{\substack{X\leq_\Lambda L\\ X\simeq M\\ \Lambda'X=Z}}[(L\cap Z)/X], \qquad  P_{\iota}(L,M;Z;s):=\sum_{\substack{X\leq_\Lambda L\\ X\simeq M\\ \Lambda'X=Z}}|(L\cap Z)/X|^{-s},\\
    Q_{\Lambda'}(W;s):=\sum_{U\le_{\Lambda'} W}\mu(U,W)|W/U|^{-s}.
\end{gather*}
\end{scholium}

\begin{corollary}\label{corollary_last_effective_formula}
    Keep the assumptions and notation of Theorem \ref{effective_solomons_first_conjecture}. Then we have the formulae
    \begin{align*}
        \frac{Z_\Lambda(L)}{\iota^*Z_{\Lambda'}(L')}&=\sum_{\substack{X\leq_{\Lambda} L\\ IL'\subseteq\Lambda'X}}[L/X]\iota^*Q_{\Lambda'}(\Lambda'X/IL')\\
        &=\sum_{\substack{Z\le_{\Lambda'} L'\\ IL'\subseteq Z}}[L/(L\cap Z)]P_\iota(L;Z)\iota^*Q_{\Lambda'}(Z/IL'),\\
        \frac{\zeta_\Lambda(L;s)}{\zeta_{\Lambda'}(L';s)}&=\sum_{\substack{X\leq_{\Lambda} L\\ IL'\subseteq\Lambda'X}}|L/X|^{-s}Q_{\Lambda'}(\Lambda'X/IL';s)\\
        &=\sum_{\substack{Z\le_{\Lambda'} L'\\ IL'\subseteq Z}}|L/(L\cap Z)|^{-s}P_\iota(L;Z;s)Q_{\Lambda'}(Z/IL';s),
    \end{align*}
   where, for $W=Z/IL'$, we define
\begin{gather*}
    P_{\iota}(L;Z):=\sum_{\substack{X\leq_\Lambda L\\ \Lambda'X=Z}}[(L\cap Z)/X], \qquad  P_{\iota}(L;Z;s):=\sum_{\substack{X\leq_\Lambda L\\ \Lambda'X=Z}}|(L\cap Z)/X|^{-s},\\
    Q_{\Lambda'}(W;s):=\sum_{U\le_{\Lambda'} W}\mu(U,W)|W/U|^{-s}.
\end{gather*}
\end{corollary}

\section{An explicit formula for lattices over certain orders}\label{section_application}

Let $R$ be the ring of integers in a nonarchimedean local field $K$. Let $\Lambda$ be an $R$-order in a finite-dimensional semisimple $K$-algebra $A$. Let $L$ be a left $\Lambda$-lattice. Let $\Lambda'$ be a maximal $R$-order in $A$ containing $\Lambda$, and let $\iota:\Lambda\hookrightarrow\Lambda'$ be the inclusion. Put $L'=\Lambda'L$ and $I=\{x\in A: \Lambda'x\Lambda'\subseteq\Lambda\}$. 

Let $\kappa=\mathbb{F}_q$, and let $\kappa_1,\kappa_2$ be copies of $\kappa$. Suppose that $I=J(\Lambda')=J(\Lambda)$, and that the embedding $\Lambda/I\hookrightarrow\Lambda'/I$ is isomorphic to the diagonal embedding $\kappa\hookrightarrow\kappa_1\times\kappa_2$. Put $$M=L'/IL', \qquad M_1=\kappa_1M, \qquad M_2=\kappa_2M, \qquad N=L/IL.$$
Noting that $IL'=IL$, so that $N$ is a $\kappa$-submodule of $M$, we put
$$m_1=\dim_{\kappa_1}(M_1), \qquad m_2=\dim_{\kappa_2}(M_2), \qquad c=\dim_\kappa(M/N).$$
Since $(\kappa_1\times\kappa_2)N=M$, it follows from Goursat's lemma for modules \cite[p.~171]{Lambek1966} that there exists a $\kappa$-module $C$ and epimorphisms $\rho_1:M_1\twoheadrightarrow C$, $\rho_2:M_2\twoheadrightarrow C$ such that $$N= M_1\times_C M_2:=\{(x_1,x_2)\in M_1\oplus M_2: \rho_1(x_1)=\rho_2(x_2)\}.$$ 
In particular, $c=\dim_\kappa(C)\le\min(m_1,m_2)$.

\begin{proposition}\label{prop_P}
    Let $Z$ be a finite-index $\Lambda'$-submodule of $L'$ containing $IL'$, and let $j=\dim_\kappa(\frac{Z}{L\cap Z})$. Then $j\le c$. Recall from Corollary \ref{corollary_last_effective_formula} the polynomial $P_\iota(L;Z;s)$. If $(\kappa_1\times\kappa_2)\frac{L\cap Z}{IL'}\ne \frac{Z}{IL'}$, then we have
    $$P_\iota(L;Z;s)=0.$$
    Otherwise, if $(\kappa_1\times\kappa_2)\frac{L\cap Z}{IL'}=\frac{Z}{IL'}$, then we have the formula
    $$P_\iota(L;Z;s)=\sum_{\ell=0}^{\min(m_1-j,m_2-j)}\binom{m_1-j}{\ell}_q\binom{m_2-j}{\ell}_q |\GL_\ell(\mathbb{F}_q)|q^{\ell(j-s)}.$$
    Here, $\binom{\alpha}{\beta}_q$ is the Gaussian binomial coefficient enumerating the $\beta$-dimensional subspaces of $\mathbb{F}_q^{\oplus\alpha}$.
\end{proposition}
\begin{proof}
We have $Z/(L\cap Z)\cong (L+Z)/L\subseteq L'/L\cong M/N$ as $\kappa$-modules, and so $j\le c$. Now, write $V=(L\cap Z)/IZ$ and $W=Z/IZ$. Then we have 
$$P_{\iota}(L;Z;s)=\sum_{\substack{U\leq_\kappa V\\ (\kappa_1\times\kappa_2)U=W}}\left|V/U\right|^{-s}.$$
Moreover, we have the equivalences
\begin{align*}
    (\kappa_1\times\kappa_2)V=W&\iff \Lambda'(L\cap Z)+IZ=Z \iff \Lambda'(L\cap Z)=Z \\ &\iff \Lambda'(L\cap Z)+IL'=Z\iff(\kappa_1\times\kappa_2)\frac{L\cap Z}{IL'}=\frac{Z}{IL'}.
\end{align*}
The second equivalence follows from $I=J(\Lambda')$ and Nakayama's lemma, while the third equivalence follows from $IL'=IL\subseteq L\cap Z$.

If $(\kappa_1\times\kappa_2)\frac{L\cap Z}{IL'}\neq\frac{Z}{IL'}$, then $(\kappa_1\times\kappa_2)U\subseteq (\kappa_1\times\kappa_2)V\ne W$ for every
$\kappa$-submodule $U$ of $V$. In this case, the sum defining $P_{\iota}(L;Z;s)$ is empty and thus $P_{\iota}(L;Z;s)=0$.

Suppose now that $(\kappa_1\times\kappa_2)\frac{L\cap Z}{IL'}=\frac{Z}{IL'}$, so that $(\kappa_1\times\kappa_2)V=W$. Write $W_1=\kappa_1W$, $W_2=\kappa_2W$. Since $Z\simeq L'$ as left $\Lambda'$-modules, we have $W\simeq M$ as $\kappa_1\times \kappa_2$-modules, and so we have $\dim_{\kappa_1}(W_1)=m_1$, $\dim_{\kappa_2}(W_2)=m_2$. By Goursat's lemma for modules, there exists a $\kappa$-module $B$ and epimorphisms $\psi_1:W_1\twoheadrightarrow B$, $\psi_2:W_2\twoheadrightarrow B$ such that $$V=W_1\times_B W_2:=\{(x_1,x_2)\in W_1\oplus W_2: \psi_1(x_1)=\psi_2(x_2)\}.$$
Note that $\dim_\kappa(B)=\dim_\kappa(W/V)=\dim_\kappa(Z/(L\cap Z))=j$.

The $\kappa$-submodules $U$ of $V$ satisfying $(\kappa_1\times\kappa_2)U=W$ are parameterised by triples $(U_1,U_2,\varphi)$ consisting of a $\kappa$-submodule $U_1$ of $\ker\psi_1$, a $\kappa$-submodule $U_2$ of $\ker\psi_2$, and a $\kappa$-isomorphism $\varphi:W_1/U_1\xrightarrow{\sim} W_2/U_2$ such that the following diagram commutes.
\[\begin{tikzcd}
	{W_1/U_1} & {} & {W_2/U_2} \\
	{W_1/\ker\psi_1} && {W_2/\ker\psi_2} \\
	& B
	\arrow["\varphi"', from=1-1, to=1-3]
	\arrow[two heads, from=1-1, to=2-1]
	\arrow[two heads, from=1-3, to=2-3]
	\arrow["{\bar{\psi}_1}", from=2-1, to=3-2]
	\arrow["{\bar{\psi}_2}"', from=2-3, to=3-2]
\end{tikzcd}\]
Indeed, given a triple $(U_1,U_2,\varphi)$, we put 
$$U=\{(x_1,x_2)\in W_1\oplus W_2: \varphi(x_1+U_1)=x_2+U_2\}.$$ 
Conversely, given $U$, we put
$$U_1=\{x_1\in W_1: (x_1,0)\in U\}, \qquad U_2=\{x_2\in W_2: (0,x_2)\in U\},$$
and $\varphi:W_1/U_1\xrightarrow{\sim} W_2/U_2$ is given by Goursat's lemma.

Suppose that $(U_1,U_2,\varphi)$ is such a triple. The commutativity of the diagram implies that $\dim_\kappa(\ker\psi_1/U_1)=\dim_\kappa(\ker\psi_2/U_2)=\ell$, say. Note that $\dim_\kappa(\ker\psi_i)=\dim_\kappa(W_i)-\dim_\kappa(B)=m_i-j$. Therefore, $\ell\le\min(m_1-j,m_2-j)$.

For such a fixed $\ell$, there are
$$\binom{m_i-j}{\ell}_q$$
choices for $U_i$.

Now, fix such a pair $(U_1,U_2)$. We observe that the quotients $W_i/U_i$ are extensions
of $B$ by $\ell$-dimensional $\kappa$-modules. After choosing splittings,
an admissible $\kappa$-isomorphism $\varphi:W_1/U_1\xrightarrow{\sim} W_2/U_2$ is determined by an element of
$\GL_\ell(\kappa)$ and an element of
$\Hom_\kappa(B,\kappa^\ell)$. Hence, the number of such isomorphisms is $|\GL_\ell(\kappa)|q^{j\ell}$.

Note that $V=W_1\times_B W_2$ and $U$ is similarly a fibre product of $W_1,W_2$ over a $j+\ell$-dimensional $\kappa$-module; $\dim_\kappa(W_i/U_i)=\dim_\kappa(W_i/\ker\psi_i)+\dim_\kappa(\ker\psi_i/U_i)=j+\ell$. Therefore, we find $\dim_\kappa(V/U)=(j+\ell)-j=\ell$, and thus $|V/U|^{-s}=q^{-\ell s}$.

Our formula for $P_\iota(L;Z;s)$ now follows from these considerations.
\end{proof}

\begin{proposition}\label{prop_Q}
Let $Z$ be a finite-index $\Lambda'$-submodule of $L'$ containing $IL'$. Let $$V=(L\cap Z)/IL', \qquad W=Z/IL', \qquad W_1=\kappa_1W, \qquad W_2=\kappa_2W.$$ 
Put $j=\dim_\kappa(Z/(L\cap Z))=\dim_\kappa(W/V)$, and assume that $(\kappa_1\times\kappa_2)V=W$. Then there exists a $j$-dimensional $\kappa$-submodule $B$ of $C$ such that $\rho_1(W_1)=B=\rho_2(W_2)$. Write $d_i=\dim_\kappa(W_i)-j$ and $n_i=m_i-c$. Then we have the formulae
$$|L/(L\cap Z)|^{-s}=q^{-(c-j)s}\cdot q^{-(n_1-d_1)s}\cdot q^{-(n_2-d_2)s}$$
and
$$Q_{\Lambda'}(W;s)=\prod_{k=0}^{j-1}(1-q^{k-s})^2\cdot \prod_{k=0}^{d_1-1}(1-q^k q^{j-s}) \cdot \prod_{k=0}^{d_2-1}(1-q^k q^{j-s}).$$
\end{proposition}
\begin{proof}
Since $(\kappa_1\times\kappa_2)V=W$ and $$V=N\cap W=\{(x_1,x_2)\in W_1\oplus W_2: \rho_1(x_1)=\rho_2(x_2)\},$$ we have a $\kappa$-submodule $B$ of $C$ such that $\rho_1(W_1)=B=\rho_2(W_2)$. Moreover, we have $V=W_1\times_B W_2$ and thus also $\dim_\kappa(B)=\dim_\kappa(W/V)=j$.

Since $N=M_1\times_C M_2$, we have $\dim_\kappa(N)=m_1+m_2-c$. Similarly, we have $\dim_\kappa(V)=(d_1+j)+(d_2+j)-j$. Recalling again that $IL=IL'$, we compute 
\begin{align*}
    \dim_\kappa(L/(L\cap Z))&=\dim_\kappa(L/IL)-\dim_\kappa((L\cap Z)/IL')\\
    &=\dim_\kappa(N)-\dim_\kappa(V)\\
    &=(m_1+m_2-c)-((d_1+j)+(d_2+j)-j)\\
    &=(c-j) + (m_1-c-d_1) + (m_2-c-d_2)\\
    &=(c-j) + (n_1-d_1)+ (n_2-d_2),
\end{align*}
so that 
$$|L/(L\cap Z)|^{-s}=q^{-(c-j)s}\cdot q^{-(n_1-d_1)s}\cdot q^{-(n_2-d_2)s}.$$

    Now, let $\mu_i$ be the M\"obius function of the poset of $\kappa_i$-submodules of $W_i$. Then  
    $$Q_{\Lambda'}(W;s)=\sum_{U\le_{\kappa_1\times\kappa_2} W}\mu(U,W)|W/U|^{-s}=\prod_{i=1}^2 \sum_{U_i\le_{\kappa_i} W_i}\mu_i(U_i,W_i)|W_i/U_i|^{-s}.$$
    By a result of Hall (see \cite[\S.~5, Example~2]{Rota}), 
$$\mu_i(U_i,W_i)=(-1)^\ell q^{\ell(\ell-1)/2}$$
for every $\kappa_i$-submodule $U_i$ of $W_i$ with $\dim_{\kappa_i}(W_i/U_i)=\ell$. Making use of $q$-binomial coefficients and Cauchy's $q$-binomial theorem \cite[\S.~3.3]{And84}, we see that 
$$\sum_{U_i\le W_i}\mu_i(U_i,W_i)|W_i/U_i|^{-s}=\sum_{\ell=0}^{d_i+j} \binom{d_i+j}{\ell}_{q} (-1)^\ell q^{\ell(\ell-1)/2} q^{-\ell s}=\prod_{k=0}^{d_i+j-1}(1-q^{k-s}).$$
Finally, to arrive at the desired formula for $Q_{\Lambda'}(W;s)$, we simply write
$$\prod_{k=0}^{d_i+j-1}(1-q^{k-s})=\prod_{k=0}^{j-1}(1-q^{k-s})\cdot \prod_{k=0}^{d_i-1}(1-q^k q^{j-s}).$$
\end{proof}

\begin{theorem}\label{theorem_application}
    Let $R$ be the ring of integers in a nonarchimedean local field $K$. Let $\Lambda$ be an $R$-order in a finite-dimensional semisimple $K$-algebra $A$. Let $L$ be a left $\Lambda$-lattice. Let $\Lambda'$ be a maximal $R$-order in $A$ containing $\Lambda$. Put $L'=\Lambda'L$ and $I=\{x\in A: \Lambda'x\Lambda'\subseteq\Lambda\}$. Let $\kappa=\mathbb{F}_q$, and let $\kappa_1,\kappa_2$ be copies of $\kappa$. Suppose that $I=J(\Lambda')=J(\Lambda)$, and that the embedding $\Lambda/I\hookrightarrow\Lambda'/I$ is isomorphic to the diagonal embedding $\kappa\hookrightarrow\kappa_1\times\kappa_2$. Let $M=L'/IL'$, $M_1=\kappa_1M$, $M_2=\kappa_2M$, $N=L/IL$. Noting that $IL'=IL$, so that $L/IL$ is a $\Lambda/I$-submodule of $L'/IL'$, we let $m_1=\dim_{\kappa_1}(M_1)$, $m_2=\dim_{\kappa_2}(M_2)$, $c=\dim_\kappa(M/N)$. Write $\binom{\alpha}{\beta}_q$ for the Gaussian binomial coefficient enumerating the $\beta$-dimensional subspaces of $\mathbb{F}_q^{\oplus\alpha}$. Then we have the formulae
$$\zeta_{\Lambda'}(L';s)=\prod_{h=0}^{m_1-1}(1-q^{h-s})^{-1}\cdot\prod_{i=0}^{m_2-1}(1-q^{i-s})^{-1}$$
and
    $$\frac{\zeta_\Lambda(L;s)}{\zeta_{\Lambda'}(L';s)}=\sum_{j=0}^c \binom{c}{j}_q q^{-(c-j)s}P^{m_1,m_2}_j(q,s)\prod_{k=0}^{j-1}(1-q^{k-s})^2,$$
    where
    $$P^{m_1,m_2}_j(q,s)=\sum_{\ell=0}^{\min(m_1-j,m_2-j)}\binom{m_1-j}{\ell}_q\binom{m_2-j}{\ell}_q |\GL_\ell(\mathbb{F}_q)|q^{\ell(j-s)}.$$
    Moreover, we have $c\le\min(m_1,m_2)$.
\end{theorem}
\begin{proof}
The finite-index $\Lambda'$-submodules of $L'$ are all isomorphic as left $\Lambda'$-modules, and we have the semisimple decompositions
    $$\Lambda'/J(\Lambda')\simeq \kappa_1\times\kappa_2, \qquad L'/J(\Lambda')L'\simeq \kappa_1^{\oplus m_1}\oplus \kappa_2^{\oplus m_2}.$$
Therefore, by the author's abstract Hey's formula \cite[Prop.~A.1]{Lynch2026}, we have
$$Z_{\Lambda'}(L')=\prod_{h=0}^{m_1-1}(1-q^{h}[\kappa_1])^{-1}\cdot\prod_{i=0}^{m_2-1}(1-q^{i}[\kappa_2])^{-1}.$$
Applying the continuous ring homomorphism $\mathbb{Z}\llbracket [\kappa_1],[\kappa_2]\rrbracket\to\mathbb{Z}\llbracket q^{-s}\rrbracket: [\kappa_1],[\kappa_2]\mapsto q^{-s}$, we obtain the stated formula for $\zeta_{\Lambda'}(L';s)$.

Now, we are left to apply the following formula from Corollary \ref{corollary_last_effective_formula}.
$$\frac{\zeta_\Lambda(L;s)}{\zeta_{\Lambda'}(L';s)}=\sum_{\substack{Z\le_{\Lambda'} L'\\ IL'\subseteq Z}}|L/(L\cap Z)|^{-s}P_\iota(L;Z;s)Q_{\Lambda'}(Z/IL';s)$$
By Proposition \ref{prop_P}, we have
\begin{align*}
    \frac{\zeta_\Lambda(L;s)}{\zeta_{\Lambda'}(L';s)}&=\sum_{j=0}^c\sum_{\substack{Z\le_{\Lambda'} L'\\ IL'\subseteq Z \\ \dim_\kappa(Z/(L\cap Z))=j}}|L/(L\cap Z)|^{-s}P_\iota(L;Z;s)Q_{\Lambda'}(Z/IL';s)\\
    &=\sum_{j=0}^c\sum_{\substack{Z\le_{\Lambda'} L'\\ IL'\subseteq Z\\ \dim_\kappa(Z/(L\cap Z))=j \\ (\kappa_1\times\kappa_2)\frac{L\cap Z}{IL'}=\frac{Z}{IL'}}}|L/(L\cap Z)|^{-s}P_j^{m_1,m_2}(q,s)Q_{\Lambda'}(Z/IL';s).
\end{align*}
Now, let us consider the inner sum over $Z$. Following Proposition \ref{prop_Q}, we write $W=Z/IL'$, $W_1=\kappa_1W$, $W_2=\kappa_2 W$. By Goursat's lemma, there exists a $c$-dimensional $\kappa$-module $C$ and epimorphisms $\rho_1:M_1\twoheadrightarrow C$, $\rho_2:M_2\twoheadrightarrow C$ such that $N=M_1\times_C M_2$. In particular, we have $c\le \min(m_1,m_2)$. Then, by the proposition, we have a $j$-dimensional $\kappa$-submodule $B$ of $C$ such that $\rho_1(W_1)=B=\rho_2(W_2)$. 

Write $n_i=m_i-c=\dim_\kappa(\ker\rho_i)$. For a fixed $j$-dimensional $\kappa$-submodule $B$ of $C$ and fixed $d_i=\dim_\kappa(W_i)-j$, the number of choices of $W_i$ is equal to 
$$\binom{n_i}{d_i}_q q^{j(n_i-d_i)}.$$
Indeed, we first choose the $d_i$-dimensional $\kappa$-subspace $D_i=W_i\cap\ker\rho_i$ of $\ker\rho_i$. Then, once $D_i$ is fixed, the possible $W_i$ correspond to splittings of the short exact sequence of $\kappa$-modules $0\to \ker\rho_i/D_i\to \rho_i^{-1}(B)/D_i\to B\to 0$. The set of such splittings is a torsor under $\Hom_\kappa(B,\ker\rho_i/D_i)$, and therefore has cardinality $q^{j(n_i-d_i)}$.

Hence, by Proposition \ref{prop_Q}, the contribution to $\zeta_\Lambda(L;s)/\zeta_{\Lambda'}(L';s)$ of a fixed $j$-dimensional $\kappa$-submodule $B$ of $C$ is equal to 
$$q^{-(c-j)s} P_j^{m_1,m_2}(q,s)\prod_{k=0}^{j-1}(1-q^{k-s})^2 \prod_{i=1}^2 \sum_{d_i=0}^{n_i}\binom{n_i}{d_i}_q q^{(n_i-d_i)(j-s)}\prod_{k=0}^{d_i-1}(1-q^k q^{j-s}).$$
By \cite[Prop.~5.2]{simeonov}, we have the identity
$$\sum_{d_i=0}^{n_i}\binom{n_i}{d_i}_q q^{(n_i-d_i)(j-s)}\prod_{k=0}^{d_i-1}(1-q^k q^{j-s})=1.$$
Therefore, the above contribution simplifies to 
$$q^{-(c-j)s} P_j^{m_1,m_2}(q,s)\prod_{k=0}^{j-1}(1-q^{k-s})^2.$$
Summing over all $B$ and $j$, we obtain
\begin{align*}
    \frac{\zeta_\Lambda(L;s)}{\zeta_{\Lambda'}(L';s)}=\sum_{j=0}^c \binom{c}{j}_q q^{-(c-j)s} P_j^{m_1,m_2}(q,s)\prod_{k=0}^{j-1}(1-q^{k-s})^2,  
\end{align*}
as desired.
\end{proof}

\begin{example}
    Let $\Lambda=\mathbb{Z}_p[\mathbb{Z}/p\mathbb{Z}]\simeq\mathbb{Z}_p\times_{\mathbb{F}_p}\mathbb{Z}_p[\omega]$, where $p$ is a prime number and $\omega$ is a primitive $p$-th root of unity. Let $L$ be a $\Lambda$-lattice, and let $a,b,c$ be the unique nonnegative integers such that $L\simeq \mathbb{Z}_p^{\oplus a}\oplus \mathbb{Z}_p[\omega]^{\oplus b} \oplus \Lambda^{\oplus c}$. Then we may apply Theorem \ref{theorem_application} with $q=p$, $m_1=a+c$ and $m_2=b+c$, thus recovering Theorem \ref{intro_group_algebra}.
\end{example}

\section{Declaration of generative AI and AI-assisted technologies in the manuscript preparation process}

During the preparation of this work, the author used ChatGPT to improve the readability of the manuscript. After using this tool, the author reviewed and edited the content as necessary and takes full responsibility for the final published content.

\printbibliography

\Addresses

\end{document}